\documentclass[12pt,reqno]{amsart}
\usepackage{color}
\usepackage[english]{babel}
\usepackage{amsfonts}
\usepackage{amsmath}
\usepackage{amsthm}
\usepackage{amssymb}
\usepackage[misc]{ifsym}
\usepackage{bbm}
\usepackage{mathrsfs}
\usepackage{esint}
\usepackage{faktor}
\usepackage[utf8]{inputenc}
\usepackage{afterpage}
\usepackage[left=2.9cm,right=2.9cm,top=2.8cm,bottom=2.8cm]{geometry}
\usepackage{graphicx}
\usepackage[pagewise,displaymath,mathlines]{lineno}
\usepackage{enumitem}
\usepackage[colorlinks=true,urlcolor=blue,citecolor=blue,linkcolor=blue,linktocpage,pdfpagelabels,bookmarksnumbered,bookmarksopen]{hyperref}

\newcommand{\N}{\mathbb{N}}

\newcommand{\R}{\mathbb{R}}

\newcommand{\Div}{\mathrm{div}\,}
\newcommand{\dx}{\,{\rm d} x}

\renewcommand{\phi}{\varphi}

\newtheorem{lemma}{Lemma}[section]
\newtheorem{thm}[lemma]{Theorem}
\newtheorem{prop}[lemma]{Proposition}

\theoremstyle{definition}
\newtheorem{defi}[lemma]{Definition}
\newtheorem{rmk}[lemma]{Remark}
\newtheorem{ex}[lemma]{Example}

\numberwithin{equation}{section}

\begin{document}
\title[Quasilinear Dirichlet systems with competing operators and convection ]{Quasilinear Dirichlet systems with competing operators and convection}
\author[L. Gambera]{Laura Gambera}
\address[L. Gambera]{Dipartimento di Matematica e Informatica, Universit\`a degli Studi di Catania, Viale A. Doria 6, 95125 Catania, Italy}
\email{laura.gambera@unipa.it}
\author[S.A. Marano]{Salvatore A. Marano}
\address[S.A. Marano]{Dipartimento di Matematica e Informatica, Universit\`a degli Studi di Catania, Viale A. Doria 6, 95125 Catania, Italy}
\email{marano@dmi.unict.it}
\author[D. Motreanu]{Dumitru Motreanu}
\address[D. Motreanu]{Department of Mathematics, University of Perpignan, 66860 Perpignan, France;
College of Science, Yulin Normal University, Yulin, People’s Republic of China}
\email{motreanu@univ-perp.fr}
\maketitle
\begin{abstract}
In this paper, we consider a quasi-linear Dirichlet system with possible competing $(p,q)$-Laplacians and convections. Due to the lack of ellipticity, monotonicity, and variational structure, the standard approaches to the existence of weak solutions cannot be adopted. Nevertheless, through an approximation procedure and a corollary of Brouwer's fixed point theorem we show that the problem admits a solution in a suitable sense.
\end{abstract}
\let\thefootnote\relax
\section{Introduction and main result}
Let $\Omega\subseteq\R^N$, $N\geq 2$, be a bounded domain with a Lipschitz boundary $\partial \Omega$, let $1<q_i<p_i<N$, and let $\mu_i\in\R$, $i=1,2$. Consider the system
\begin{equation}\label{prob}\tag{P}
\left\{
\begin{alignedat}{2}
-\Delta_{p_1} u+\mu_1\Delta_{q_1} u & =f_1(x,u,v,\nabla u,\nabla v) && \quad\mbox{in}\;\;\Omega,\\
-\Delta_{p_2} v+\mu_2\Delta_{q_2} v & =f_2(x,u,v,\nabla u,\nabla v) && \quad\mbox{in}\;\;\Omega,\\
u=v & =0 && \quad\mbox{on}\;\;\partial\Omega,
\end{alignedat}
\right.
\end{equation}
where $\Delta_r u:= \Div(|\nabla u|^{r-2}\nabla u)$, with $1<r<+\infty$, denotes the so-called $r$-Laplacian while $f_i:\Omega\times\R^2\times\R^{2N}\to\R$, $i=1,2$, are Carathéodory functions satisfying appropriate conditions; see \ref{growthcondition}--\ref{Nboundeness} below. Since $q_i<p_i$, we seek solutions $(u,v)$ to \eqref{prob} in the space
\begin{equation}\label{defX}
X:=W^{1,p_1}_0(\Omega) \times W^{1,p_2}_0(\Omega).
\end{equation}
The involved differential operator 
$$A_{i,\mu_i}(w):=-\Delta_{p_i} w+\mu_i\Delta_{q_i}w,\;\; w\in W^{1,p_i}_0(\Omega),$$
exhibits completely different behaviors according to whether
$$\mu_i<0,\quad\mu_i=0,\quad\mbox{or}\;\;\mu_i>0.$$
If $\mu_i<0$ then $A_{i,\mu_i}$ is basically patterned \cite[Section 2.1]{MaMo} after the well known $(p_i,q_i)$-Laplacian, namely $\Delta_{p_i}w+\Delta_{q_i}w$, which turns out non-homogeneous because $p_i\neq q_i$. When $\mu_i=0$ it coincides with the classical negative $p_i$-Laplacian. Both cases have been widely investigated and meaningful results are by now available in the literature; see, e.g.,  \cite[Chapter 6]{M2018}. On the contrary, for $\mu_i>0$ the operator $A_{i,\mu_i}$ contains the \textit{difference} between the $q_i$-Laplacian and the $p_i$-Laplacian. It is usually called competitive and, as already pointed out in \cite{ML,M}, doesn't comply with any ellipticity or monotonicity condition. In fact,
given $u_0\in W^{1,p_i}_0(\Omega)\setminus\{0\}$ and chosen $u:=tu_0$, $t>0$, the expression
$$\langle A_{i,\mu_i}(u),u\rangle=t^p\Vert\nabla u_0\Vert_p^p
-\mu_i t^q\Vert u_0\Vert_q^q$$
turns out negative for $t$ small and positive when $t$ is large. Hence, if $\mu_i>0$, problem \eqref{prob} gathers at least two challenging technical features:
\begin{itemize}
\item driving operators are neither elliptic nor monotone. This prevents to use surjectivity results for pseudo-monotone operators, nonlinear regularity theory, besides comparison principles.
\item right-hand sides depend on the gradient of solutions, so that variational techniques cannot be enforced directly.
\end{itemize}
To overcome these difficulties we first exploit the Galerkin method, thus working in a sequence $\{X_n\}$ of finite dimensional sub-spaces of $X$. A \textit{finite dimensional approximate solution} $(u_n,v_n)\in X_n$ to \eqref{prob} is obtained for every $n\in\N$ via a corollary of the Brouwer fixed point theorem. Next, an appropriate passage to the limit yields a solution in a generalized sense; cf. Definitions \ref{Def1}--\ref{Def2}.

The idea of weakening the notion of solution to treat more general situations is classical; cf. for instance \cite[p. 183]{S} and \cite{Z1,Z2}.

As far as we know, equations driven by competing operators have previously been investigated only in \cite{ML,M}. The main difference between them concerns right-hand sides. The paper \cite{ML} doesn't treat convective reactions. Solutions are thus obtained through  Galerkin's method and Ekeland's variational principle. In \cite{M}, instead, any variational approach is forbidden, because the nonlinearity depends on the gradient of solutions. 

This work continues the study started in \cite{M}, treating \textit{convective systems driven by competing operators}. Section 2, beyond some auxiliary results, contains the notions of generalized and strong generalized solution to \eqref{prob}, adapted from Definitions 2.3 and 2.5, respectively, in \cite{M}. A sequence of finite dimensional approximate solutions is constructed in Section 3. Finally, Section 4 deals with a limit procedure that provides a generalized solution to \eqref{prob}. Under a growth condition stronger than \ref{growthcondition} (see \ref{growthcondition2} below) we next show that \eqref{prob} admits a strong generalized solution, which becomes a weak one as soon as $\mu_1\vee\mu_2<0$. The section ends by exhibiting a natural class of reactions that fulfill our hypotheses; cf. Example \ref{finex}.
\section{Preliminaries and hypotheses}
Let $Y$ be a real Banach space and let $Y^*$ be its topological dual, with duality brackets $\langle\cdot,\cdot\rangle$. An operator $B:Y\to Y^*$ is said to be:
\begin{itemize}
\item \emph{bounded} when it maps bounded sets into bounded sets.
\item \emph{monotone} if $\langle B(y)-B(z),y-z\rangle\geq 0$ for all $y,z\in Y$.
\item \emph{pseudo-monotone} when $y_n\rightharpoonup y$ in $Y$ and $\displaystyle{\limsup_{n\to\infty}}\langle B(y_n),y_n-y\rangle\leq 0$ imply $$\liminf_{n\to\infty}\langle B(y_n),y_n-z\rangle\geq\langle B(y),y-z\rangle\;\;\forall\, z\in Y.$$
\item \emph{of type $(\mathrm{S})_+$} provided
\begin{equation*}
y_n\rightharpoonup y\;\;\mbox{in $Y$,}\;\;\limsup_{n\to\infty}\langle B(y_n),y_n-y\rangle\leq 0\implies y_n\to y\;\;\mbox{in $Y$.}   
\end{equation*}
\end{itemize}
The following consequence of Brouwer's fixed point theorem will play a basic role in the sequel. For the proof we refer to \cite[p. 37]{S}.
\begin{thm}\label{brouwer}
Let $(Y,\|\cdot\|)$ a finite dimensional normed space and let $B:Y\to Y^*$ be continuous. Assume there exists $R>0$ such that
$$\langle B(y),y\rangle\ge 0\quad\mbox{for all $y\in Y$ with $\|y\|=R$.}$$ 
Then the equation $B(y)=0$ admits a solution $\bar y\in Y$ fulfilling $\|\bar y\|\le R$.
\end{thm}

The symbol $|E|$ stands for the $N$-dimensional Lebesgue measure of the set $E\subseteq\R^{N}$ and
\begin{equation*}
a\vee b:=\max\{a,b\},\;\; a\wedge b:=\min\{a,b\}\;\;\quad\forall\, a,b\in\R.
\end{equation*}
Given a real number $r>1$, set $r':=r/(r-1)$. If $r<N$ then $r^*:=Nr/(N-r)$. We denote by $\Vert\cdot\Vert_r$ the usual norm of $L^r(\Omega)$, while $\Vert\cdot\Vert_{1,r}$ indicates the norm of $W^{1,r}_0(\Omega)$ coming from Poincaré's inequality, i.e.,
\begin{equation*}
\Vert u\Vert_{1,r}:=\Vert\nabla u\Vert_r\, ,\quad u\in W^{1,r}_0(\Omega).
\end{equation*}
The symbol $W^{1,-r'}(\Omega)$ stands for the dual space of $W^{1,r}_0(\Omega)$ while $\lambda_{1,r}$ is the first eigenvalue of the operator $-\Delta_r$ in $W^{1,r}_0(\Omega)$. One has
\begin{equation}\label{5}
\lambda_{1,r}=\inf_{u\in W^{1,r}_0(\Omega)\setminus\{0\}}\frac{\|u\|_{1,r}^r}{\| u\|_r^r}>0.
\end{equation}
The following facts are known; see, e.g., \cite[Theorem 9.16]{B} and \cite{K}.
\begin{prop}\label{embeddingresult}
Let $1<r<N$. Then the embedding:
\begin{itemize}
\item[{\rm (a)}] $W^{1,r}_0(\Omega)\hookrightarrow W^{1,s}_0(\Omega )$ is continuous for all $s\in[1,r]$.
\item[{\rm (b)}] $W^{1,r}_0(\Omega)\hookrightarrow L^s(\Omega)$ is continuous for every $s\in[1,r^*]$ and compact when $s<r^*$.
\item[{\rm (c)}] $L^{s'}(\Omega)\hookrightarrow W^{1,-r'}(\Omega)$ is continuous for all $s\in[1,r^*]$.
\end{itemize} 
\end{prop}
By (a) the differential operator $-\Delta_{p_i}+\mu_i\Delta_{q_i}$ ($i=1,2$) turns out  well posed on $W^{1,p_i}_0(\Omega)$ because
$$1 <q_i<p_i<N\implies W^{1,p_i}_0(\Omega)\hookrightarrow W^{1,q_i}_0(\Omega).$$
Accordingly, we will seek solutions $(u,v)$ to system \eqref{prob} lying in the product space $X$ given by \eqref{defX},
endowed with the usual norm
\begin{equation*}
\|(u,v)\|:= \|u\|_{1,p_1}+\|v\|_{1,p_2}\, .
\end{equation*}
It is evident that $X$ is a separable Banach space and one has
$$X^*=W^{1,-p_1'}(\Omega)\times W^{1,-p_2'}(\Omega).$$
Let $A_{\mu_1,\mu_2}: X\to X^*$ defined by
\begin{equation*}
\begin{split}
\langle A_{\mu_1,\mu_2}(u,v),(\phi,\psi)\rangle
=\int_{\Omega} & \big(|\nabla u|^{p_1-2}\nabla u
-\mu_1|\nabla u|^{q_1-2}\nabla u\big)\nabla\phi\dx\\
& +\int_{\Omega}\big(|\nabla v|^{p_2-2}\nabla v
-\mu_2|\nabla v|^{q_2-2}\nabla v\big)\nabla\psi\dx
\end{split}
\end{equation*}
for all $(u,v),(\phi,\psi)\in X$. A simple argument, which exploits standard properties \cite{Pe} of the $r$-Laplacian, shows that $A_{\mu_1,\mu_2}$ turns out bounded and continuous. Moreover, if $\mu_1\vee\mu_2<0$ then $A_{\mu_1,\mu_2}$ is also monotone.

The growth conditions below will be posited on the reactions $f_1,f_2:\Omega\times\R^2 \times\R^{2N}\to\R$.

\begin{enumerate}[label={${\rm (H_1)}$}]
\hypertarget{H1}{}
\item\label{growthcondition}
There exist constants $C_i>0$ and functions $\sigma_i\in L^{(p_i^{*})'}(\Omega)$, $i=1,2$, such that
\begin{equation*}
\begin{split}
&|f_1(x,s,t,\xi,\nu)|\le
C_1\left[|s|^{p_1^*-1}+|t|^{\frac{p_2^*}{(p_1^*)'}}
+|\xi|^{\frac{p_1}{(p_1^*)'}}+|\nu|^{\frac{p_2}{(p_1^*)'}}\right]+\sigma_1(x),\\
&|f_2(x,s,t, \xi,\nu)|\le
C_2\left[|s|^{\frac{p_1^*}{(p_2^*)'}}+|t|^{p_2^*-1}
+|\xi|^{\frac{p_1}{(p_2^*)'}}+|\nu|^{\frac{p_2}{(p_2^*)'}}\right]+\sigma_2(x)
\end{split}
\end{equation*}
a.e. in $\Omega$ and for every $s,t\in\R$, $\xi,\nu\in\R^N$.
\end{enumerate}
\begin{rmk}
Since $p-1<\frac{p}{(p^{*})'}$ whatever $p>1$, the growth rate $\frac{p}{(p^*)'}$ of gradient terms in \ref{growthcondition} is better than the analogous one of \cite{M}, namely $p-1$.
\end{rmk}
Hereafter, $C$, $\hat C$, $\tilde C$, etc. will denote generic positive constants, which may change explicit value from line to line.
\begin{prop}\label{P1}
Under \ref{growthcondition}, the Nemytskii operator $\mathcal{N}_{f_i}: X \to
L^{(p_i^*)'}(\Omega)$ defined by
\begin{equation*}
\mathcal{N}_{f_i}(u,v):=f_i(\cdot,u,v,\nabla u,\nabla v)\quad \forall\, (u,v)\in X
\end{equation*}
is well posed, bounded, and continuous. 
\end{prop}
\begin{proof}
Let $i=1$ (the case $i=2$ can be treated similarly). If $(u,v)\in X$ then $u\in L^{p_1}(\Omega)$, $v\in L^{p_2}(\Omega)$, $\nabla u\in (L^{p_1}(\Omega))^N$, and $\nabla v\in (L^{p_2}(\Omega))^N$. Through \ref{growthcondition} we have
\begin{equation*}
|f_1(\cdot,u,v,\nabla u,\nabla v)|\le
C_1\left[|u|^{p_1^*-1}+|v|^{\frac{p_2^*}{(p_1^*)'}} +|\nabla u|^{\frac{p_1}{(p_1^*)'}} +|\nabla v|^{\frac{p_2}{(p_1^*)'}}\right]+\sigma_1
\end{equation*}
a.e. in $\Omega$, whence
\begin{equation}\label{wellpos}
\begin{split}
&\int_\Omega |f_1(\cdot,u,v,\nabla u,\nabla v)|^{(p_1^{*})'}\dx\\
&\le\ C\left[\int_{\Omega}(|u|^{p_1^*}+|v|^{p_2^*} +|\nabla u|^{p_1}
+|\nabla v|^{p_2})\dx
+\int_{\Omega}\sigma_1^{(p_1^*)'}\dx\right]<+\infty
\end{split}
\end{equation}  
thanks to conclusion (b) of Proposition \ref{embeddingresult}. This means that the function $\mathcal{N}_{f_1}$ is well defined. From \eqref{wellpos} it also follows
\begin{equation*}
\begin{split}
\|\mathcal{N}_{f_1}(u,v)\|_{(p_1^*)'} & \le\hat{C}\left(\|u\|_{p_1^*}^{p_1^*}
+\|v\|_{p_2^*}^{p_2^*}+\|\nabla u\|_{p_1}^{p_1}+\|\nabla v\|_{p_2}^{p_2}
+1\right)^{\frac{1}{(p_1^*)'}}\\
&\le \tilde{C}\left(\|u\|_{p_1^*}^{p_1^*-1}+\|v\|_{p_2^*}^{\frac{p_2^*}{(p_1^*)'}}
+\|\nabla u\|_{p_1}^{\frac{p_1}{(p_1^*)'}}+\|\nabla v\|_{p_2}^{\frac{p_2}{(p_1^*)'}}
+1\right),\quad (u,v)\in X.
\end{split}	
\end{equation*}
So, by Proposition \ref{embeddingresult} again, $\mathcal{N}_{f_1}$ maps bounded sets into bounded sets. Finally, since \ref{growthcondition} holds,  a classical result on superposition operators (cf. \cite[Theorem 2.3]{DF}) yields the continuity of $\mathcal{N}_{f_1}$.
\end{proof}
\begin{rmk}
Because of (c) in Proposition \ref{embeddingresult} the function $\mathcal{N}_{f_i}:X \to W^{1,-p_i'}(\Omega)$ enjoys the same properties.   
\end{rmk}
To investigate problem \eqref{prob} we will consider the operator $\mathcal{A}_{\mu_1,\mu_2}:X\to X^*$ given by
\begin{equation}\label{defcalA}
\mathcal{A}_{\mu_1,\mu_2}(u,v)= A_{\mu_1,\mu_2}(u,v)-(\mathcal{N}_{f_1}(u,v),\mathcal{N}_{f_2}(u,v))
\quad\forall\, (u,v)\in X.
\end{equation}
Proposition \ref{P1} shows that if \ref{growthcondition} is satisfied then $\mathcal{A}_{\mu_1,\mu_2}$ turns out well posed, bounded, and continuous for any $(\mu_1,\mu_2) \in\R^2$. Moreover, weak solutions $(u,v)\in X$ of \eqref{prob} comply with the equation
\begin{equation}\label{weaksol}
\mathcal{A}_{\mu_1,\mu_2}(u,v)=0,
\end{equation}
and vice-versa. However, \eqref{weaksol} cannot be solved via the classical suriectivity theorem for pseudo-monotone operators once $\mu_1\vee\mu_2>0$. In fact, $u\mapsto\Delta_p u$ is not pseudo-monotone even when $p=2$; see \cite[p. 1511]{M}.

We will overcome this difficulty by introducing  weaker notions of solution, adapted from those in \cite{M}.
\begin{defi}\label{Def1}
Suppose \ref{growthcondition} holds. A pair $(u,v)\in X$ is  called a \textit{generalized solution} to \eqref{prob} provided there exists a sequence $\{(u_n,v_n)\}\subseteq X$ such that
\begin{itemize}
\item[(a)] $(u_n,v_n)\rightharpoonup (u,v)$ in $X$,
\item[(b)] $\mathcal{A}_{\mu_1,\mu_2}(u_n,v_n)\rightharpoonup 0$ in $X^*$, and
\item[(c)] $\displaystyle{\lim_{n\to\infty}}\langle\mathcal{A}_{\mu_1,\mu_2}(u_n,v_n),(u_n-u,v_n-v)\rangle=0$.
\end{itemize}
\end{defi}
Specifically, this means that:
\begin{equation*}
u_n\rightharpoonup u\;\:\mbox{in $W^{1,p_1}_0(\Omega)$}\quad\mbox{and}\quad  v_n\rightharpoonup v\;\;\mbox{in $W^{1,p_2}_0(\Omega)$}; 
\end{equation*}
\begin{equation*}
\left\{
\begin{aligned}
-\Delta_{p_1} u_n +\mu_1\Delta_{q_1}u_n-\mathcal{N}_{f_1}(u_n,v_n)\rightharpoonup 0\quad \mbox{in $W^{-1,p_1'}(\Omega)$,}\\
-\Delta_{p_2} v_n +\mu_2\Delta_{q_2}v_n-\mathcal{N}_{f_2}(u_n,v_n) \rightharpoonup 0\quad \mbox{in $W^{-1,p_2'}(\Omega)$;}
\end{aligned}
\right.
\end{equation*}
\begin{equation*}
\left\{\begin{aligned}
\langle-\Delta_{p_1} u_n+\mu_1\Delta_{q_1} u_n, u_n-u\rangle-
\int_\Omega\mathcal{N}_{f_1}(u_n,v_n)(u_n-u)\dx\to 0,\\
\langle-\Delta_{p_2} v_n+\mu_2\Delta_{q_2} v_n, v_n-v\rangle-
\int_\Omega\mathcal{N}_{f_2}(u_n,v_n)(v_n-v)\dx\to 0.
\end{aligned}
\right.
\end{equation*}
If we strengthen \ref{growthcondition} as follows:
\begin{enumerate}[label={${\rm (H_1')}$}]
\hypertarget{H1'}{}
\item\label{growthcondition2}
For appropriate $r_i, s_i\in]1,p_i^*[$, $D_i>0$, and $\sigma_i\in L^{(s_i)'}(\Omega)$, $i=1,2$, one has
\begin{equation*}
\begin{split}
|f_1(x,s,t,\xi,\nu)|\le D_1\left[|s|^{\frac{p_1^*}{r_1'}}+|t|^{\frac{p_2^*}{r_1'}}
+|\xi|^{\frac{p_1}{r_1'}}+|\nu|^{\frac{p_2}{r_1'}}\right]+\sigma_1(x), \\
|f_2(x,s,t,\xi,\nu)|\le D_2\left[|s|^{\frac{p_1^*}{r_2'}}+|t|^{\frac{p_2^*}{r_2'}} +|\xi|^{\frac{p_1}{r_2'}}+|\nu|^{\frac{p_2}{r_2'}}\right]+\sigma_2(x)
\end{split}
\end{equation*}
a.e. in $\Omega$ and for every $s,t\in\R$, $\xi,\nu\in\R^N$,
\end{enumerate}
then the definition below can be posited. It should be noted that \ref{growthcondition2} implies \ref{growthcondition}. In fact, from $r_i<p_i^*$ it follows $r_i'>(p_i^*)'$, whence
$$\frac{q}{r_i'}<\frac{q}{(p_i^*)'}\quad\forall\, q\in\{p_1^*,p_2^*,p_1,p_2\}.$$
\begin{defi}\label{Def2}
Assume \ref{growthcondition2} holds. A pair $(u,v)\in X$ is called  a \textit{strong generalized solution} to \eqref{prob} provided there exists a sequence $\{(u_n,v_n)\} \subseteq X$ such that (a)--(b) of Definition \ref{Def1} are fulfilled and, moreover,
\begin{itemize}
\item[$({\rm c}')$] $\displaystyle{\lim_{n\to\infty}}\langle A_{\mu_1,\mu_2}(u_n,v_n),
(u_n-u,v_n-v)\rangle=0$.
\end{itemize}
\end{defi}
\begin{rmk} Until now, we have considered tree kinds of solution: weak, generalized, and strong generalized. One has
$$\mbox{weak}\implies\mbox{strong generalized}\implies\mbox{generalized.}$$
To see the first implication, simply pick $(u_n,v_n):=(u,v)$ for all $n\in\N$. The other can be easily verified arguing as in the proof of Theorem \ref{T2}.
\end{rmk}
We will establish the existence of generalized (respectively, strong generalized) solutions under hypothesis \ref{growthcondition} (respectively, \ref{growthcondition2}) and the next one.
%
\begin{enumerate}[label={${\rm (H_2)}$}]
\hypertarget{H2}{}
\item \label{Nboundeness}
There are constants $c_i,d_i>0$, $i=1,2$, with
\begin{equation}\label{7}
c_1+c_2+(d_1+d_2)\frac{1}{\lambda_{1,p_1}\wedge\lambda_{1,p_2}}<1,
\end{equation}
and functions $\gamma_i\in L^1(\Omega)$ satisfying
\begin{equation*}
\begin{split}
f_1(x,s,t,\xi,\nu)s\le c_1\left(|\xi|^{p_1}+|\nu|^{p_2}\right)
+d_1\left(|s|^{p_1}+|t|^{p_2}\right)+\gamma_1(x),\\
f_2(x,s,t,\xi,\nu)t\le c_2\left(|\xi|^{p_1}+|\nu|^{p_2}\right)
+d_2\left(|s|^{p_1}+|t|^{p_2}\right)+\gamma_2(x)
\end{split}
\end{equation*}
a.e. in $\Omega$ and for all $s,t\in\R$, $\xi,\nu\in\R^N$.
\end{enumerate}
\section{Approximate solutions via Galerkin's method}
Since the Banach space $X$ defined in \eqref{defX} is separable, it admits a Galerkin basis. So, there exists a sequence $\{X_n\}$ of vector sub-spaces of $X$ such that
\begin{itemize}
\item $\dim(X_n)<\infty\;\;\forall\, n\in\N$,
\item $X_n\subseteq X_{n+1}\;\;\forall\, n\in\N$, and
\item $\overline{\bigcup_{n=1}^{\infty} X_n}=X$.
\end{itemize}
Evidently, we may suppose $X_n=U_n\times V_n$, with  $U_n\subseteq W^{1,p_1}_0(\Omega)$ and
$V_n\subseteq W^{1,p_2}_0(\Omega)$.
\begin{prop}\label{galerkin}
Let \hyperlink{H1}{${\rm (H_1)}$}--\hyperlink{H2}{${\rm (H_2)}$} be satisfied and let $(\mu_1, \mu_2)\in\R^2$.
Then for every $n\in\N$ there exists a pair $(u_n,v_n)\in X_n$ such that
\begin{equation}\label{galestimates}
\begin{split} 
&\langle-\Delta_{p_1} u_n +\mu_1\Delta_{q_1}u_n,\phi\rangle- 
\int_\Omega\mathcal{N}_{f_1}(u_n,v_n)\phi\dx =0\\
&\langle-\Delta_{p_2} v_n +\mu_2\Delta_{q_2}v_n,\psi\rangle-
\int_\Omega\mathcal{N}_{f_2}(u_n,v_n)\psi\dx =0
\end{split}
\quad\forall\, (\phi,\psi)\in X_n.
\end{equation}
Moreover, the sequence $\{(u_n,v_n)\}$ is bounded in $X$.
\end{prop}
\begin{proof}
Fix $n\in\N$. To shorten notation, write $B:=\mathcal{A}_{\mu_1,\mu_2} \lfloor_{X_n}$. Setting $B:=(B_1,B_2)$ one has
\begin{equation*}
\begin{split}
&\langle B_1(u,v),\phi\rangle=
\langle-\Delta_{p_1} u+\mu_1\Delta_{q_1}u,\phi\rangle
-\int_\Omega\mathcal{N}_{f_1}(u,v)\phi\dx \\
&\langle B_2(u,v),\psi\rangle=
\langle-\Delta_{p_2} v+\mu_2\Delta_{q_2}v,\psi\rangle
-\int_\Omega\mathcal{N}_{f_2}(u,v)\psi\dx
\end{split}
\quad\forall\, (u,v),(\phi,\psi)\in X_n.
\end{equation*}
\textit{Claim:} There exists $R>0$ such that
\begin{equation*}
(\phi,\psi)\in X_n,\;\;\Vert(\phi,\psi)\Vert=R\implies
\langle B(\phi,\psi),(\phi,\psi)\rangle\geq 0.
\end{equation*}
In fact, \ref{Nboundeness} and \eqref{5} imply
\begin{equation}\label{mupos}
\begin{split}
\int_\Omega & \mathcal{N}_{f_1}(\phi,\psi)\phi\dx+
\int_\Omega\mathcal{N}_{f_2}(\phi,\psi)\psi\dx\\
&\leq\int_\Omega c_1\left(|\nabla\phi|^{p_1}+|\nabla\psi|^{p_2}\right)\dx
+\int_\Omega d_1\left(|\phi|^{p_1}+|\psi|^{p_2}\right)\dx+\|\gamma_1\|_1\\
&\phantom{PP}
+\int_\Omega c_2\left(|\nabla\phi|^{p_1}+|\nabla\psi|^{p_2}\right)\dx+
\int_\Omega d_2\left(|\phi|^{p_1}+|\psi|^{p_2}\right)\dx+\|\gamma_2\|_1\\
&\leq 
c_1\left(\|\phi\|^{p_1}_{1,p_1}+\|\psi\|^{p_2}_{1,p_2}\right)
+d_1\left(\lambda_{1,p_1}^{-1}\|\phi\|^{p_1}_{1,p_1}
+\lambda_{1,p_2}^{-1}\|\psi\|^{p_2}_{1,p_2}\right)\\
&\phantom{PP}
+c_2\left(\|\phi\|^{p_1}_{1,p_1}+\|\psi\|^{p_2}_{1,p_2}\right)
+d_2\left(\lambda_{1,p_1}^{-1}\|\phi\|^{p_1}_{1,p_1}
+\lambda_{1,p_2}^{-1}\|\psi\|^{p_2}_{1,p_2}\right)+
\|\gamma_1\|_1+\|\gamma_2\|_1\\
&=\left[c_1+c_2+(d_1+d_2)\lambda_{1,p_1}^{-1}\right]
\|\phi\|^{p_1}_{1,p_1}
+\left[c_1+c_2+(d_1+d_2)\lambda_{1,p_2}^{-1}\right]
\|\psi\|^{p_2}_{1,p_2}+C,
\end{split}
\end{equation}
where $C:=\Vert\gamma_1\Vert_1+\Vert\gamma_2\Vert_1$. Since $q_i<p_i$, using H\"older's inequality we thus obtain
\begin{equation*}
\begin{split}
\langle & B(\phi,\psi),(\phi,\psi)\rangle\\
&\geq\left[1-c_1-c_2-(d_1+d_2)\lambda_{1,p_1}^{-1}\right]
\|\phi\|^{p_1}_{1,p_1}
+\left[1-c_1-c_2-(d_1+d_2)\lambda_{1,p_2}^{-1}\right]
\|\psi\|^{p_2}_{1,p_2}\\
&\phantom{PP}-|\mu_1||\Omega|^{\frac{p_1-q_1}{p_1}} \|\phi\|^{q_1}_{1,p_1}
-|\mu_2||\Omega|^{\frac{p_2-q_2}{p_2}}\|\psi\|^{q_2}_{1,p_2}-C,\\
\end{split}
\end{equation*}
and the claim easily follows from \eqref{7}.\\
Now, Theorem  \ref{brouwer} gives a pair $(u_n,v_n)\in X_n$ such that $B(u_n,v_n)=0$, which entails \eqref{galestimates}. Finally, letting $(\phi,\psi)=(u_n,v_n)$ in \eqref{galestimates} and arguing as done for \eqref{mupos} produces
\begin{equation*}
\begin{split}
\| u_n\|^{p_1}_{1,p_1}+ & \| v_n\|^{p_2}_{1,p_2}\\
& \le |\mu_1|\|u_n\|^{q_1}_{1,q_1}+|\mu_2|\|v_n\|^{q_2}_{1,q_2}
+\int_\Omega\mathcal{N}_{f_1}(u_n,v_n)u_n\dx+\int_\Omega\mathcal{N}_{f_2}(u_n,v_n)v_n\dx\\
&\le |\mu_1||\Omega|^{\frac{p_1-q_1}{p_1}} \| u_n\|^{q_1}_{1,p_1}
+|\mu_2||\Omega|^{\frac{p_2-q_2}{p_2}}\| v_n\|^{q_2}_{1,p_2}\\
&\phantom{PPPP}+\left[c_1+c_2+(d_1+d_2)
(\lambda_{1,p_1}\wedge\lambda_{1,p_2})^{-1}\right]
\left(\| u_n\|^{p_1}_{1,p_1}+\| v_n\|^{p_2}_{1,p_2}\right)+C
\end{split}
\end{equation*}
for every $n\in\N$. At this point, the boundedness of $\{(u_n,v_n)\} \subseteq X$ is an immediate consequence of $q_i<p_i$, $i=1,2$, and \eqref{7}.
\end{proof}
\section{Existence of generalized solutions}
\begin{thm}\label{T1}
Under hypotheses \hyperlink{H1}{${\rm (H_1)}$}--\hyperlink{H2}{${\rm (H_2)}$}, for every $(\mu_1,\mu_2)\in\R^2$ problem \eqref{prob} possesses a solution in the sense of Definition \ref{Def1}.
\end{thm}
\begin{proof}
Fix $(\mu_1,\mu_2)\in\R^2$. Proposition \ref{galerkin} provides a bounded sequence $\{(u_n,v_n)\}\subseteq X$ fulfilling \eqref{galestimates}. Since $X$ is reflexive, up to sub-sequences one has $(u_n,v_n) \rightharpoonup (u,v)$ in $X$, i.e., condition (a) of Definition
\ref{Def1} holds. Moreover, the boundedness of the operator $\mathcal{A}_{\mu_1,\mu_2}$ given by \eqref{defcalA} yields a pair $(\eta_1,\eta_2)\in X^*$ such that
\begin{equation}\label{weakcalA}
\mathcal{A}_{\mu_1,\mu_2}(u_n,v_n)=
A_{\mu_1,\mu_2}(u_n,v_n)-(\mathcal{N}_{f_1}(u_n,v_n), \mathcal{N}_{f_2}(u_n,v_n))\rightharpoonup(\eta_1,\eta_2) \quad\mbox{in $X^*$.}
\end{equation}
If $(\phi,\psi)\in\cup_{n=1}^\infty X_n$ then $(\phi,\psi)$ belongs to $X_n=U_n\times V_n$ for any $n$ large enough. From \eqref{galestimates} and \eqref{weakcalA} it thus follows, after letting $n\to\infty$,
\begin{equation*}
\langle\eta_1,\phi\rangle=\langle\eta_2,\psi\rangle=0.
\end{equation*}
This entails $\eta_1=\eta_2=0$, because $(\phi,\psi)$ was arbitrary, $\overline{\cup_{n=1}^\infty U_n}= W^{1,p_1}_0(\Omega)$, and $\overline{\cup_{n=1}^\infty V_n} =W^{1,p_2}_0(\Omega)$. Therefore, assertion (b) of Definition \ref{Def1} is true. Through \eqref{weakcalA} we next have
\begin{equation}\label{bcond}
\begin{split}
&\langle-\Delta_{p_1} u_n+\mu_1\Delta_{q_1}u_n
-\mathcal{N}_{f_1}(u_n,v_n),u\rangle\to 0,\\
&\langle-\Delta_{p_2} v_n +\mu_2\Delta_{q_2}v_n
-\mathcal{N}_{f_2}(u_n,v_n),v\rangle\to 0
\end{split}
\end{equation}
while \eqref{galestimates} yields
\begin{equation}\label{galerkconb}
\begin{split}
&\| u_n\|^{p_1}_{1,p_1}=\mu_1\| u_n\|^{q_1}_{1,q_1}
+\int_\Omega\mathcal{N}_{f_1}(u_n,v_n)u_n\dx, \\
&\| v_n\|^{p_2}_{1,p_2}=\mu_2\| v_n\|^{q_2}_{1,q_2}+
\int_\Omega\mathcal{N}_{f_2}(u_n,v_n)v_n\dx
\end{split}
\quad\forall\, n\in\N.
\end{equation}
Gathering \eqref{bcond} and \eqref{galerkconb}
together produces
\begin{equation}\label{c_1}
\left\{
\begin{aligned}
&\langle-\Delta_{p_1} u_n +\mu_1\Delta_{q_1}u_n-\mathcal{N}_{f_1}(u_n,v_n),u_n-u\rangle\to 0,\\
&\langle-\Delta_{p_2} v_n +\mu_2\Delta_{q_2}v_n-\mathcal{N}_{f_2}(u_n,v_n),v_n-v\rangle\to 0.
\end{aligned}
\right.
\end{equation}
So, condition (c) of Definition \ref{Def1} holds, which means that $(u,v)$ turns out a generalized solution to \eqref{prob}.
\end{proof}
\begin{thm}\label{T2}
Let \hyperlink{H1'}{${\rm (H_1')}$}--\hyperlink{H2}{${\rm (H_2)}$} be satisfied and let $(\mu_1,\mu_2)\in\R^2$. Then problem \eqref{prob} admits a solution $(u,v)\in X$ in the sense of Definition \ref{Def2}. Moreover, $(u,v)$ is a weak solution once $\mu_1\vee\mu_2<0$.
\end{thm}
\begin{proof}
Since \ref{growthcondition2} implies \ref{growthcondition}, the same arguments adopted to prove Theorem \ref{T1} furnish both a bounded sequence $\{(u_n, v_n)\}\subseteq X$ and $(u, v)\in X$ satisfying (a)--(b) of Definition \ref{Def2} as well as \eqref{c_1}. Thus, it remains to verify $({\rm c}')$. Thanks to \ref{growthcondition2} and H\"older's inequality we obtain
\begin{equation*}
\begin{split}
& \left| \int_\Omega\mathcal{N}_{f_1}(u_n,v_n)(u_n-u)\dx\right|\\
&\le D_1\int_\Omega\left(|u_n|^{\frac{p_1^*}{r_1'}} +|v_n|^{\frac{p_2^*}{r_1'}}+|\nabla u_n|^{\frac{p_1}{r_1'}}
+|\nabla v_n|^{\frac{p_2}{r_1'}}\right)|u_n-u|\dx +\int_\Omega\sigma_1|u_n-u|\dx\\
&\le D_1\left(\|u_n\|^{r_1'}_{p_1^*}+\|v_n\|^{r_1'}_{p_2^*}
+\|u_n\|^{r_1'}_{1,p_1}+\|v_n\|^{r_1'}_{1,p_2}\right)
\|u_n-u\|_{r_1}+\|\sigma_1\|_{s_1'}\|u_n-u\|_{s_1}\\
&\le C\|u_n-u\|_{r_1}+\|\sigma_1\|_{s_1'}\|u_n-u\|_{s_1}
\end{split}	
\end{equation*}
for all $n\in\N$, because $\{(u_n,v_n)\}$ is bounded. The condition $r_1\vee s_1<p_1^*$ then entails $u_n\to u$ in $L^{r_1}(\Omega)\cap L^{s_1}(\Omega)$, where a sub-sequence is considered when necessary; cf. Proposition \ref{embeddingresult}. Therefore,
\begin{equation}\label{c'f1}
\lim_{n\to\infty}\int_\Omega\mathcal{N}_{f_1}(u_n,v_n)(u_n-u)\dx=0.
\end{equation}
Analogously, one has
\begin{equation}\label{c'f2}
\lim_{n\to\infty}\int_\Omega\mathcal{N}_{f_2}(u_n,v_n)(u_n-u)\dx=0.
\end{equation}
From \eqref{c_1}--\eqref{c'f2} it follows
\begin{equation*}
\lim_{n\to\infty}\langle-\Delta_{p_1} u_n+\mu_1\Delta_{q_1}u_n, u_n-u\rangle
=\lim_{n\to\infty}\langle-\Delta_{p_2} v_n+\mu_2\Delta_{q_2}v_n, v_n-v\rangle=0,
\end{equation*}
i.e., $({\rm c}')$ of Definition \ref{Def2} holds, and $(u,v)$ is a strong generalized solution to \eqref{prob}.

Finally, let $\mu_1\vee\mu_2<0$. We will show that $(u,v)$ fulfills \eqref{weaksol}, that is
\begin{equation}\label{cthm2}
\begin{split} 
&\langle-\Delta_{p_1} u-|\mu_1|\Delta_{q_1}u,\phi\rangle
-\int_\Omega\mathcal{N}_{f_1}(u,v)\phi\dx=0,\\
&\langle-\Delta_{p_2} v -|\mu_2|\Delta_{q_2}v,\psi\rangle
-\int_\Omega\mathcal{N}_{f_2}(u,v)\psi\dx=0
\end{split}
\quad\forall\, (\phi,\psi)\in X,
\end{equation}
which completes the proof. The monotonicity of $-|\mu_i| \Delta_{q_i}$ and $({\rm c}')$ yield
\begin{equation*}
\begin{split}
&\limsup_{n\to\infty}\langle-\Delta_{p_1}u_n,u_n-u\rangle \le\lim_{n\to\infty}\langle-\Delta_{p_1} u_n-|\mu_1| \Delta_{q_1}u_n, u_n-u\rangle=0,\\
&\limsup_{n\to\infty}\langle-\Delta_{p_2}v_n,v_n-v\rangle
\le\lim_{n\to\infty}\langle-\Delta_{p_2}v_n-|\mu_2| \Delta_{q_2}v_n,v_n-v\rangle=0.
\end{split}	
\end{equation*}
Since the operator $-\Delta_{p_i}$ is of type $({\rm S})_+$, we infer $(u_n,v_n)\to (u,v)$ in $X$. Now, combining the continuity of $\mathcal{A}_{\mu_1,\mu_2}$ with condition (b) in Definition \ref{Def1} one has
\begin{equation*}
\mathcal{A}_{\mu_1,\mu_2}(u,v) =\lim_{n\to\infty}\mathcal{A}_{\mu_1,\mu_2}(u_n,v_n)=0,
\end{equation*}
namely \eqref{cthm2} holds true.
\end{proof}
The next example provides a natural class of reactions $f_1,f_2$ that fulfill \hyperlink{H1}{${\rm (H_1)}$}--\hyperlink{H2}{${\rm (H_2)}$}.
\begin{ex}\label{finex}
Let $f_i:\Omega\times\R^2\times\R^{2N}\to \R$, $i=1,2$, be given by
\begin{equation*}
\begin{split}
&f_1(x,s,t,\xi,\nu):=|s|^{\alpha_1-2}s
+\frac{s}{s^2+1}\left(|t|^{\frac{p_2}{(p_1^*)'}} +|\xi|^{\beta_1}+|\nu|^{\frac{p_2}{(p_1^*)'}}+h_1(x)\right),\\
&f_2(x,s,t,\xi,\nu):=|t|^{\alpha_2-2}t
+\frac{t}{t^2+1}\left(|s|^{\frac{p_1}{(p_2^*)'}} +|\xi|^{\frac{p_1}{(p_2^*)'}}+|\nu|^{\beta_2}+h_2(x)\right),
\end{split}
\end{equation*}
where $1\leq\alpha_i<p_i$, $1\leq\beta_i<\frac{p_i}{(p_i^*)'}$, and $h_i\in L^{(p_i^*)'}(\Omega)$. Observe that
\begin{equation*}
\begin{split}
|f_1(x,s,t,\xi,\nu)| & \leq |s|^{\alpha_1-1}
+|t|^{\frac{p_2}{(p_1^*)'}}+|\xi|^{\beta_1}
+|\nu|^{\frac{p_2}{(p_1^*)'}}+|h_1(x)|\\
&\leq |s|^{p_1^*-1}+|t|^{\frac{p_2^*}{(p_1^*)'}} +|\xi|^{\frac{p_1}{(p_1^*)'}}+|\nu|^{\frac{p_2}{(p_1^*)'}} +|h_1(x)|+3
\end{split}
\end{equation*}
because $\alpha_1<p_1<p_1^*$, $p_2<p_2^*$,
$\beta_1<\frac{p_1}{(p_1^*)'}$. Similarly,
\begin{equation*}
\begin{split}
|f_2(x,s,t,\xi,\nu)| & \leq |s|^{\frac{p_1}{(p_2^*)'}}
+|t|^{\alpha_2-1}+|\xi|^{\frac{p_1}{(p_2^*)'}} +|\nu|^{\beta_2}+|h_2(x)|\\
&\leq |s|^{\frac{p_1^*}{(p_2^*)'}}+|t|^{p_2^*-1}
+|\xi|^{\frac{p_1}{(p_1^*)'}}+|\nu|^{\frac{p_2}{(p_2^*)'}} +|h_2(x)|+3
\end{split}
\end{equation*}
as $p_1<p_1^*$, $\alpha_2<p_2<p_2^*$, $\beta_2<\frac{p_2}{(p_2^*)'}$. Hence, \ref{growthcondition} is true with $C_1:=C_2:=1$ and $\sigma_i(x):=|h_i(x)|+3$. If $c_i,d_i>0$, $i=1,2$, satisfy \eqref{7} then, with appropriate $\hat c_i>0$, one has both
\begin{equation*}
\begin{split}
f_1(x,s,t,\xi,\nu)s & \leq |s|^{\alpha_1}+|t|^{\frac{p_2}{(p_1^*)'}} +|\xi|^{\beta_1}+|\nu|^{\frac{p_2}{(p_1^*)'}}+h_1(x)\\
&\leq c_1(|\xi|^{p_1}+|\nu|^{p_2})+d_1(|s|^{p_1}+|t|^{p_2})+h_1(x)+\hat c_1,
\end{split}
\end{equation*}
since $\alpha_1<p_1$, $\frac{p_2}{(p_1^*)'}<p_2$, $\beta_1<\frac{p_1}{(p_1^*)'}<p_1$, and 
\begin{equation*}
\begin{split}
f_2(x,s,t,\xi,\nu)t & \leq |s|^{\frac{p_1}{(p_2^*)'}}+|t|^{\alpha_2}
+|\xi|^{\frac{p_1}{(p_2^*)'}}+|\nu|^{\beta_2}+h_2(x)\\
&\leq c_2(|\xi|^{p_1}+|\nu|^{p_2})+d_2(|s|^{p_1}+|t|^{p_2})+h_2(x)+\hat c_2,
\end{split}
\end{equation*}
because $\frac{p_1}{(p_2^*)'}<p_1$, $\alpha_2<p_2$, $\beta_2<\frac{p_2}{(p_2^*)'}<p_2$. Therefore, setting $\gamma_i(x)=h_i(x)+\hat c_i$, we see that also \hyperlink{H2}{${\rm (H_2)}$} holds.
\end{ex}
\section*{Acknowledgments}
\noindent L. Gambera and S.A. Marano are members of GNAMPA of INDAM.

S.A. Marano was supported by the following research projects: 1) PRIN 2017 `Nonlinear Differential Problems via Variational, Topological and Set-valued Methods' (Grant No. 2017AYM8XW) of MIUR; 2) `MO.S.A.I.C.' PRA 2020--2022 `PIACERI' Linea 2 of the University of Catania.
%

%
\end{document}